\documentclass{article}
\usepackage{amsmath, amsthm, amssymb}
\usepackage{braket}

\theoremstyle{plain}
\newtheorem{Theorem}{Theorem}[section]
\newtheorem{Lemma}[Theorem]{Lemma}

\newtheorem{Proposition}[Theorem]{Proposition}
\newtheorem{Conjecture}[Theorem]{Conjecture}
\newtheorem{Fact}[Theorem]{Fact}
\theoremstyle{remark}

\theoremstyle{definition}
\newtheorem{Example}[Theorem]{Example}
\newtheorem{Definition}[Theorem]{Definition}
\newtheorem{Algorithm}[Theorem]{Algorithm}

\DeclareMathOperator{\conv}{CONV}
\DeclareMathOperator{\lcm}{LCM}

\DeclareMathOperator{\initial}{in}

\newcommand{\monoid}[1]{\mathbf{#1}}

\newcommand{\isom}{\cong}

\newcommand{\realpart}[1]{\Re(#1)}
\newcommand{\initialord}[1]{\initial_{<}(#1)}

\newcommand{\groebner}{Gr\"obner }
\newcommand{\moebius}{M\"obius }

\title{Ehrhart Series for Connected Simple Graphs}
\author{Matsui Tetsushi\\
\small{Principles of Informatics Research Division, %\\
National Institute of Informatics.}\\
\texttt{tetsushi@nii.ac.jp}}
\begin{document}
\maketitle

\begin{abstract}
The Ehrhart ring of the edge polytope \(\mathcal{P}_G\) for 
a connected simple graph \(G\) is known to coincide with
the edge ring of the same graph if \(G\) satisfies the odd cycle condition.
This paper gives for a graph which does not satisfy the condition,
a generating set of the defining ideal of the Ehrhart ring
of the edge polytope,
described by combinatorial information of the graph.
From this result, two factoring properties of the Ehrhart series
are obtained;
the first one factors out bipartite biconnected components,
and the second one factors out a even cycle which shares only one edge
with other part of the graph.
As an application of the factoring properties, the root distribution of
Ehrhart polynomials for bipartite polygon trees is determined.
\end{abstract}

\section{Introduction}
\label{sec:intro}

\subsection{Background}
\label{sec:background}

This paper studies explicit construction and factoring properties
of Ehrhart series of edge polytope for connected simple graphs.
It is motivated by the root distribution of Ehrhart polynomials,
which is one of the current topics on computational commutative algebra.
In particular, the conjecture of Beck et al.~\cite{BDDPS2005} 
attracts much attention.
\begin{Conjecture}[Beck--De~Loera--Develin--Pfeifle--Stanley]\label{conj:dstrip}
  All roots $\alpha$ of Ehrhart polynomials of lattice $D$-polytopes
  satisfy $-D \le \realpart \alpha \le D - 1$.
\end{Conjecture}
The author contributed to a recent paper~\cite{MHNOH2011} providing some
computational evidence of the conjecture.
However, in the paper, rigorous proofs are shown only in a few cases,
because the Ehrhart polynomials are known only for a few families:
such as complete graphs or complete multipartite graphs.

The Ehrhart polynomial is always related to the Hilbert series of
a certain graded \(K\)-algebra, called the Ehrhart ring.
We call the Hilbert series of the Ehrhart ring the {\em Ehrhart series}.
The subject of this paper is those for the edge polytopes.
An edge polytope \(\mathcal{P}_G\) is an integral convex polytope defined
for a graph \(G\) (see Section~\ref{sec:edgept}).
Associated with the graph \(G\), there is a graded \(K\)-algebra \(K[G]\),
called an edge ring, which gives
the Ehrhart series for \(G\) if the algebra is normal;
the normality of \(K[G]\) is equivalent to the ``odd cycle condition''
on the graph \(G\)~\cite{OH1998,SVV1998}.
The definition of the odd cycle condition is as follows~\cite{FHM1965}:

\begin{Definition}\label{dfn:occ}
  The {\em odd cycle condition\/} is a condition for a graph \(G\) whereby
  any two odd cycles in \(G\) share a vertex or they are connected
  by a path whose length is one.
\end{Definition}

In such cases, the Ehrhart series is explicitly computable from the
\groebner basis of a toric ideal \(I_G\), called an edge ideal.
However, if \(K[G]\) is not normal, there have been no direct construction
for the Ehrhart ring for the graph \(G\);
one has to go through the edge polytope and use Brion's theorem or something
to obtain the Ehrhart polynomial.
It is hard to see the relationship between a graph and 
the corresponding Ehrhart ring.
This paper describes the Ehrhart ring directly from
combinatorial information of the graph.

\subsection{Preliminaries}
\label{sec:intrograph}

\subsubsection{Definitions of Graphs and Hypergraphs}
\label{sec:redef}

Definitions of the (hyper-)graphs presented in this section may seem
somewhat peculiar but are convenient for our purpose.
Let \(G\) be a triple \((V, E, f)\) of a finite set \(V\),
a finite set \(E\) disjoint with \(V\), and a map \(f\) from \(E\) to
a ring \(R=R(V)\).
By putting a few restrictions on \(f\), we have classes of (hyper-)graphs.
There are several useful conditions:
\begin{enumerate}
  \setlength{\itemsep}{1pt}
\item\label{G00} \(R\) is the free commutative monoid ring \(K[T;\monoid{V}]\)
with \(K\) a field;
\item\label{G01} \(R\) is the free (non-commutative) monoid ring \(K[T;V^*]\)
with \(K\) a field;
\item\label{G10} all the images of \(f\) are monomials;
\item\label{G11} all the images of \(f\) are quadratic;
\item\label{G3} all the images of \(f\) are squarefree.
\item\label{G2} \(f\) is injective;
\end{enumerate}

The {\em graphs} are obtained from conditions~\ref{G00}, \ref{G10}
and \ref{G11}.
We restrict our discussion throughout the paper to the {\em simple graphs},
which requiring \ref{G3} and \ref{G2} in addition to the conditions for graphs;
by condition~\ref{G3}, there are no loops,
and by condition~\ref{G2}, there are no multiple edges.
The {\em hypergraphs} are those triples with the conditions \ref{G00}, 
\ref{G10}, \ref{G3} and \ref{G2}.
In other words, relaxing the condition of simple graphs gives hypergraphs
that the number of vertices connected by an edge is arbitrary instead of two.
The condition~\ref{G01} instead of \ref{G00} for graphs
gives the {\em directed graphs}, though we do not use it in this paper.

In the following, we omit the dummy variables like \(T\) in \(K[T;\monoid{V}]\);
instead, we denote the ring simply as \(K[\monoid{V}]\).
Hence, \(\monoid{V}\) is understood as the free commutative monoid
generated by vertices, written multiplicatively.
Moreover, we assume that the characteristic of \(K\) is zero.

\subsubsection{Edge Polytopes and Ehrhart Polynomials}
\label{sec:edgept}

Let \(G=(V,E,\phi)\) be a graph without multiple edges.
The edge polytope \(\mathcal{P}_G\) of a simple graph \(G\)
is defined as follows.
Let the vertex set \(V\) of \(G\) be \(\set{v_1,\ldots,v_n}\), 
and let a homomorphism
\(\epsilon\) from \(\monoid{V}\), the monoid of monomials,
to \(\mathbb{Z}^n\) be defined by
\[\epsilon: \prod v_i^{m_i} \longmapsto \sum m_i \mathbf{e}_i,\]
where \(\mathbf{e}_i\) is the \(i\)-th fundamental unit vector.
Then, a map \(\rho\) from the edge set \(E\) of \(G\) to \(\mathbb{Z}^n\)
is defined as the composition \(\epsilon \circ \phi\) .
The edge polytope \(\mathcal{P}_G \subset \mathbb{R}^n\) is the convex hull
of the image of \(\rho\):
\begin{eqnarray*}
  \mathcal{P}_G &=& \conv \rho(E)\\
                &=& \Set{\sum_{e_i \in E}\lambda_i\rho(e_i) | 0\leq\lambda_i\leq 1,\ \sum_{e_i \in E}\lambda_i = 1,\ \lambda_i\in\mathbb{R}}.
\end{eqnarray*}

The Ehrhart polynomial \(i_G = i_{\mathcal{P}_G}\) of the edge polytope
\(\mathcal{P}_G\) is the counting function of the integral points in
dilated polytopes, that is,
\(i_G(m) = |m \mathcal{P}_G \cap \mathbb{Z}^n|\).
For convenience, we define \(i_G(0)=1\), and
we call the generating function
\[\sum_{m=0}^{\infty} i_G(m) t^m\]
the Ehrhart series \(H_G(t) = H_{\mathcal{P}_G}(t)\) for \(\mathcal{P}_G\).

\subsubsection{Edge Ring and Edge Ideal}
\label{sec:erei}

The material presented in this short section is not new,
but only notations may differ from standard one.
Let \(G=(V,E,\phi)\) be a connected simple graph.
The graph map \(\phi: E \longrightarrow K[\monoid{V}]\) of a graph \(G\)
is linearly extended to a \(K\)-algebra homomorphism 
\(\phi^*: K[\monoid{E}] \longrightarrow K[\monoid{V}]\).
The edge ideal \(I_G \subset K[\monoid{E}]\) is defined as:
\[I_G\ = (t - u\,|\,t\text{, }u\text{ are monomials and }\phi^*(t)=\phi^*(u)).\]
It is a homogeneous binomial ideal.
The edge ring \(K[G]\) of \(G\) is the image of \(\phi^*\):
\[K[G] = \phi^*(K[\monoid{E}]) \isom K[\monoid{E}] / \ker \phi^*,\]
and \(I_G=\ker\phi^*\).
The generators of \(I_G\) correspond to a certain class of
even closed walks on \(G\).
For example, the Graver basis of \(I_G\),
which consists of primitive (Definition~\ref{dfn:primitive})
even closed walk, is given in the following theorem of
Tatakis and Thoma~\cite{TT2010}.

\begin{Theorem}[\cite{TT2010}]\label{thm:evenclosedwalks}
  Let \(G\) be a graph and \(w\) an even closed walk of \(G\).
  The binomial \(B_w\) is primitive if and only if 
  \begin{enumerate}
  \setlength{\itemsep}{1pt}
  \item every block of \(w\) is a cycle or a cut edge,
  \item every multiple edge of the walk \(w\) is a double edge of the
    walk and a cut edge of \(w\),
  \item every cut vertex of \(w\) belongs to exactly two blocks and
    it is a sink of both.
  \end{enumerate}
\end{Theorem}

As already mentioned in Section~\ref{sec:background}, the edge ring \(K[G]\)
gives the Ehrhart series if and only if \(G\) is an edge-normal graph;
here, we mean by edge-normal graph, a graph \(G\) which satisfies
the odd cycle condition (Definition~\ref{dfn:occ}).
Thus, we investigate the non-edge-normal graphs, next.

\subsubsection{Hyperedge Ring and Hyperedge Ideal}
\label{sec:herhei}

Let \(G\) be a connected non-edge-normal graph
with fixed numbering of its odd cycles;
let \(C_i\) denote the \(i\)-th odd cycle.
We say a pair of odd cycles in a graph is an exceptional pair
if any connecting path of the cycles are of length at least two.
A set \(\Theta\) consists of symbols \(\theta_{i j}\) each corresponding
to an exceptional pair \((C_i, C_j)\) with \(i < j\).
Let \(F\) denote the union \(E \cup \Theta\), and
\(\psi\) the map extending \(\phi\),
which sends \(\theta_{i j}\) in \(F\) to the product of vertices on \(C_i\)
and \(C_j\) in \(K[\monoid{V}]\).
Then, \(\hat{G}\ = (V, F, \psi)\) is a hypergraph.
The \(K\)-algebra homomorphism \(\psi^*\) is defined from \(\psi\) similarly
to \(\phi^*\) from \(\phi\).
Accordingly, we define the hyperedge ideal \(I_{\hat{G}} \subset K[\monoid{F}]\):
\[I_{\hat{G}}\ = (t - u\,|\,t\text{, }u\text{ are monomials and }\psi^*(t)=\psi^*(u)).\]
We need a degree function on \(K[\monoid{F}]\), that is not a standard one.
\begin{Definition}\label{dfn:psidegree}
  For any monomial \(T\) in \(K[\monoid{F}]\), \(\psi^*\)-{\em degree} of \(T\)
  is half the number of vertices multiplied in the image \(\psi^*(T)\).
  Moreover, any element \(f\in K[\monoid{F}]\) is a sum 
  \(f=\sum_{i=1}^{k}c_i T_i\),
  and the degree of \(f\) is \(\max_{i=1,\ldots,k}\deg T_i\).
\end{Definition}
More precisely, each edge \(e \in E\) has degree one;
each \(\theta_{i j} \in \Theta\) has degree \(\tfrac{1}{2}(n_i + n_j)\)
where \(n_i\) (respectively \(n_j\)) is the number of vertices
in the odd cycle \(C_i\) (respectively \(C_j\)).
Then, the binomial ideal \(I_{\hat{G}}\) is homogeneous with respect to
\(\psi^*\)-degree.
The hyperedge ring \(K[\hat{G}]\) of
\(\hat{G}\) is the image of \(\psi^*\):
\[K[\hat{G}] = \psi^*(K[\monoid{F}]) \isom K[\monoid{F}]/\ker\psi^*,\]
and \(I_{\hat{G}}=\ker\psi^*\).
The \(K\)-algebra \(K[\hat{G}]\) is graded \(K\)-algebra with respect to
\(\psi^*\)-degree.
Actually, let \(K[\hat{G}]_m\) denote the \(K\)-vector space generated by degree
\(m\) elements in \(K[\hat{G}]\).
Then, \(K[\hat{G}]_i K[\hat{G}]_j \subset K[\hat{G}]_{i+j}\)
since the ideal equates only elements of the same degree.
In Section~\ref{sec:ehrhart}, we prove that the hyperedge ring \(K[\hat{G}]\)
is what we have sought for the Ehrhart series.
%Before moving to the argument about the Ehrhart series, however,
%we discuss about the generator of the hyperedge ideals.

\subsubsection{Crude Elements}
\label{sec:strprim}

The binomial \(\theta_{i j}^2 - C_i C_j\) should be in 
the hyperedge ideal \(I_{\hat{G}}\),
because the product of edges of \(C_i\) and \(C_j\) in \(K[\monoid{E}]\)
is sent by \(\phi^*\) (and \(\psi^*\)) to the square of \(\psi(\theta_{i j})\),

Suppose \(C_i\) and \(C_j\) are an exceptional pair; then there are
paths connecting the cycles, all of which have lengths at least two.
Let \(N_{i j}^{(p)}\) denote the \(p\)-th such path connecting
\(C_i\) and \(C_j\).
Moreover, if \(N_{i j}^{(p)}\) is \((e_{k_0}, e_{k_1},\ldots,e_{k_r})\),
let \(N_{i j}^{(p)+} = \prod_{l:\text{even}}e_{k_l}\) and
\(N_{i j}^{(p)-} = \prod_{l:\text{odd}}e_{k_l}\).
Similarly, \(C_i^{+}\) and \(C_i^{-}\) denote the alternating products
of edges on the cycle.
The choice of the sign, \(C_i^{+}\) or \(C_i^{-}\), depends on the sign
of \(N_{i j}^{(p)\pm}\); that is, the shared vertex of \(C_i\) and 
\(N_{i j}^{(p)}\) is incident to either edges of \(C_i^{+}\) and an edge
of \(N_{i j}^{(p)-}\) or edges of \(C_i^{-}\) and an edge of \(N_{i j}^{(p)+}\).
Then, \(\theta_{i j} N_{i j}^{(p)+} - C_i^{-} C_j^{-} N_{i j}^{(p)-}\) is
in \(I_{\hat{G}}\).

How far should we continue to count up such elements in the hyperedge ideal?
To answer the question,
this section introduces the notion of crude elements.
They form a special generating set of the hyperedge ideal \(I_{\hat{G}}\),
shown in Section~\ref{sec:basiccrude}.

\begin{Definition}\label{dfn:crude}
  For given a graded \(K\)-algebra \(R\) and a homogeneous binomial
  ideal \(I\), an element \(T - U\) of the ideal
  is {\em crude\/} if and only if 
  \(T \neq U\) and
  there are no \(T_i - U_i\) (\(i = 1,\ldots,k\)) in \(I\)
  satisfying the all of following conditions:
  \begin{enumerate}
  \setlength{\itemsep}{1pt}
  \item \(\forall i\) \(\deg (T_i) < \deg (T)\),
  \item \(T_1\) and \(U_k\) are proper divisors of \(T\) and \(U\) respectively,
  \item \(\exists V_i \in R\) (\(i=2,\ldots,k\)) such that
    \(V_i U_i = V_{i+1} T_{i+1}\) for \(i=1,\ldots,k-1\), with \(V_1 = T/T_1\).
  \end{enumerate}
\end{Definition}

The crudeness above is a tightening of the following primitiveness
on graph walks in~\cite{OH1999}, rephrased in terms of ideal:

\begin{Definition}\label{dfn:primitive}
  An element \(T - U\) of an ideal \(I\) 
  in a graded \(K\)-algebra \(R\) is {\em primitive\/}
  if and only if 
  \(T \neq U\) and
  there is no \(T_1 - U_1\) in \(I\)
  satisfying that
  \(T_1\) and \(U_1\) are proper divisors of \(T\) and \(U\) respectively.
\end{Definition}

The condition of primitiveness uses only \(1\) in place of \(k\)
in the conditions of crudeness.
Hence, if an element is crude then it is primitive.

\subsection{Structure of The Paper}
\label{sec:structure}

In Section~\ref{sec:ideal}, 
motivated by the fact that the hyperedge ring \(K[\hat{G}]\) is the Ehrhart
ring for a non-edge-normal graph \(G\) (Proposition~\ref{prop:ring}),
we prove the main theorem that the
crude elements generates the hyperedge ideal.

\setcounter{section}{2}
\setcounter{Theorem}{7}
\begin{Theorem}
  The following elements form a generating set of \(I_{\hat{G}}\):
  \begin{enumerate}
  \setlength{\itemsep}{1pt}
  \item a set of crude generators of \(I_G\);
  \item \(\theta_{i j}^2 - C_i C_j\) for any \(\theta_{i j} \in \Theta\);
  \item \(\theta_{i j} N_{i j}^{(p)\pm} - C_i^{\mp} C_j^{\mp} N_{i j}^{(p)\mp}\) for any \(\theta_{i j} \in \Theta\) and \(N_{i j}^{(p)\pm}\) without \(N_{i j}^{(q)\pm}\) which properly divides \(N_{i j}^{(p)\pm}\);
  \item \(\theta_{i j} N_{j k}^{(p)\pm} C_k^{\pm} - \theta_{i k} N_{j k}^{(p)\mp} C_j^{\mp}\) for any \(\theta_{i j} \in \Theta\) and \(N_{i j}^{(p)\pm}\) without \(N_{i j}^{(q)\pm}\) which properly divides \(N_{i j}^{(p)\pm}\);
  \item \(\theta_{i j}\theta_{k l} - \tilde{\theta}_{i k}\tilde{\theta}_{j l}\) for any \(\theta_{i j}, \theta_{k l} \in \Theta\) with \(i,j,k,l\) are different each other; and
  \item \(\theta_{i j}\theta_{i k} - \tilde{\theta}_{j k}C_i\) and \(\theta_{i j}\theta_{l j} - \tilde{\theta}_{i l}C_j\) for any \(\theta_{i j}, \theta_{i k}, \theta_{l j} \in \Theta\).
  \end{enumerate}
  Here, \(\tilde{\theta}_{i j}\) means either \(\theta_{i j}\) if \(C_i\) and \(C_j\) are an exceptional pair or \(C_i^{\pm}C_j^{\pm}\) otherwise.
\end{Theorem}

We also provide an algorithm in Section~\ref{sec:algorithm}
to compute the Ehrhart polynomial from a given connected simple graph.

In Section~\ref{sec:factoring}, we present two factoring properties
of the Ehrhart series, both of which based on the algorithm and
properties of \moebius sums on lcm-lattices.

\setcounter{section}{3}
\setcounter{Theorem}{2}
\begin{Theorem}[First Factoring Property]
  The Ehrhart series \(H_G\) of a graph \(G\) has a factorization
  \[H_G(t) = H_{G_0}(t) \prod_{i=1}^{r'} H_{B_i}(t),\]
  where \(G_0\), \(B_1\), \ldots, \(B_{r'}\) are the biconnected decomposition
  of \(G\) with oddments.
\end{Theorem}

\setcounter{Theorem}{4}
\begin{Theorem}[Second Factoring Property]
  Let \(G\) be a connected graph and \(G^{(1)}\) and \(G^{(2)}\) be its subgraphs.
  Assume 
  (1) each edge of \(G\) belongs either \(G^{(1)}\) or \(G^{(2)}\),
  except exactly one edge \(e\) which is shared by both;
  (2) \(G^{(2)}\) is a bipartite graph; and
  (3) \(e\) is a part of a cycle in \(G^{(2)}\).
  Then, the Ehrhart series \(H_G(t)\) can be factored as
  \[H_G(t) = H_{G^{(1)}}(t) \left(H_{G^{(2)}}(t) (1-t)\right).\]
\end{Theorem}

Finally, Section~\ref{sec:ept} applies the lemma of 
Rodriguez-Villegas~\cite{Rod2002} to obtain the root distribution
of the Ehrhart polynomials for bipartite polygon trees 
(Proposition~\ref{prop:ptroots}),
whose Ehrhart series are determined by using the second factoring property.

\setcounter{section}{1}

\section{Hyperedge Ideals}
\label{sec:ideal}

\subsection{Ehrhart series}
\label{sec:ehrhart}

The Hilbert series \(H_A\) of a graded \(K\)-algebra \(A\) is:
\[H_{A}(t) = \sum_{n=0}^{\infty} \left(\dim_K{A_n}\right) t^n .\]

The Hilbert series for the \(K\)-algebra \(K[G]\) is
the Ehrhart series for \(\mathcal{P}_G\) if \(G\) is edge-normal.
Unfortunately, it differs from the Ehrhart series for \(\mathcal{P}_G\)
if \(G\) is a non-edge-normal graph.
However, we can overcome the gap.
This is the motivation to consider the hyperedge ring \(K[\hat{G}]\).

\begin{Proposition}\label{prop:ring}
  The hyperedge ring \(K[\hat{G}]\) is a
  graded \(K\)-algebra, whose Hilbert series is the Ehrhart series \(H_G(t)\)
  for edge polytope \(\mathcal{P}_G\).
\end{Proposition}
\begin{proof}
  We have already seen that \(K[\hat{G}]\) is a graded \(K\)-algebra.

  It is shown in~\cite{OH1998} that normalization of  
  \(K[G]\) can be obtained with
  the exceptional pairs of odd cycles\footnote{
    In~\cite{OH1998}, it is claimed that the only exceptional pairs
    that have no vertex in common should be considered.
    However, this is too restrictive; in fact, two exceptional pairs
    that, for example, have a cycle in common correspond to independent
    integral points in \(\mathcal{P}_G\).
  }.
  Thus, \(F\) contains all necessary elements, i.e.,
  all the integer points in \(m \mathcal{P}_G\) for any \(m\)
  are in \(\epsilon \circ \psi^*(\monoid{F})\).
  There are integer points counted multiple times in the image,
  but it is possible to count each of them only once by identifying
  the preimage of each point.
  Thus, the monomials of \(K[\hat{G}]\),
  which is isomorphic to \(K[\monoid{F}]/I_{\hat{G}}\),
  have one-to-one correspondence with integer points in 
  \(m \mathcal{P}_G\) for some \(m\).
  Because all integer points of \(m \mathcal{P}_G\) correspond
  to degree \(m\) elements of \(K[\hat{G}]\),
  the Ehrhart polynomial \(i_G(m) = \dim_K{K[\hat{G}]_m}\).
\end{proof}

This means that the Ehrhart ring for a non-edge-normal graph \(G\) is
given as a hyperedge ring \(K[\hat{G}]\) of extended hypergraph \(\hat{G}\).

We have the Ehrhart series \(H_G\) as
\[
H_G(t) = \sum_{n=0}^{\infty} \left(\dim_K{K[\hat{G}]_n}\right) t^n
= \sum_{n=0}^{\infty} i_G(n) t^n
= \frac{i_G^*(t)}{(1-t)^{D+1}},\]
where \(D=\dim\mathcal{P}_G\) and \(i_G^*(t) \in \mathbb{Z}[t]\) with 
\(\deg{i_G^*} \le D\).

\subsection{Basic Properties of Crude Elements}
\label{sec:basiccrude}

It is crucial from Proposition~\ref{prop:ring} to know
a generating system of \(I_{\hat{G}}\).
The following proposition is essential for the purpose of this section.

\begin{Proposition}\label{prop:spgenerate}
  A homogeneous binomial ideal \(I\) of a finitely generated graded
  \(K\)-algebra \(R\) can be generated by crude elements.
\end{Proposition}
\begin{proof}
  Assume \(X - Y\) is not a crude element,
  but is in a generating set \(S \subset I\).
  Hence, there exist \(X_i - Y_i\) (\(i = 1,\ldots,k\)) in \(I\) and
  \(V_i\) (\(i=2,\ldots,k\)) in \(R\)
  satisfying the conditions of Definition~\ref{dfn:crude}.
  Let \(I'=(X_i - Y_i\,|\,i = 1,\ldots,k)\).
  Then,
  \begin{eqnarray*}
    X - Y &=& (X/X_1)X_1 - Y\\
    &\equiv& (X/X_1)Y_1 - Y \pmod{I'}\\
    &=& V_2 X_2 - Y\\
    &\equiv& V_2 Y_2 - Y \pmod{I'}\\
    &\cdots&\\
    &\equiv& V_k Y_k - Y \pmod{I'}\\
    &=& (V_k - Y/Y_k) Y_k.
  \end{eqnarray*}
  Thus, \(X - Y\) is in \(I' + (V_k - Y/Y_k)\).
  In particular, \(I\) is generated by
  \[S \cup \Set{X_i - Y_i | i=1,\ldots,k} \cup \{V_k - Y/Y_k\} \setminus \{X-Y\}.\]
  Since degrees strictly decrease on every replacement and
  \(R\) is Noetherian, this process will eventually stop.
  The resulting generating set is a finite one consisting of
  crude elements.
\end{proof}

As a consequence of this proposition,
it is sufficient to consider the crude elements in \(I_{\hat{G}}\)
for giving a generating set.
In order to determine whether an element in \(I_{\hat{G}}\) is crude or not,
we prepare the following lemma.

\begin{Lemma}\label{lem:commondivisor}
  In the same situation with Proposition~\ref{prop:spgenerate},
  assume for \(T - U\) in an ideal \(I\) that
  there exist \(T_1 - U_1\) and \(T_3 - U_3\) in \(I\) 
  such that \(T_1\) and \(U_3\) properly divides \(T\) and \(U\),
  respectively, and there exists a nontrivial common divisor
  for \(U_1\) and \(T_3\).
  Then, \(T - U\) is not crude.
\end{Lemma}
\begin{proof}
  Let \(X\) be a nontrivial common divisor for \(U_1\) and \(T_3\).
  Then, \(V_1=T/T_1\) by definition leads
  \(V_1 U_1 = X V_1 (U_1/X)\), and \(X (U/U_3)(T_3/X) = (U/U_3)T_3\)
  is obvious.
  Now, since \(T-U\in I_{\hat{G}}\), \(X V_1 (U_1/X) - X (U/U_3)(T_3/X)\)
  is also in the ideal.
  However, because \(X\) is a monomial, it is not an element of the binomial
  ideal.
  Then, \(V_1 (U_1/X) - (U/U_3)(T_3/X) \in I\).
  Let \(T_2 = V_1 (U_1/X)\), \(U_2=(U/U_3)(T_3/X)\), \(V_2=X\) and \(V_3=U/U_3\).
  Verifying that \(\deg(T_2)<\deg(T)\) and other conditions is easy.
\end{proof}

The last lemma means that if part of \(T\) and part of \(U\) are
transformed by the ideal to elements having a nontrivial common divisor,
then \(T - U\) is not essential.

\subsection{Proof of the Main Theorem}
\label{sec:proofmain}

Let \(G\) be a non-edge-normal connected simple graph.
We prove the main theorem by a series of lemmata.

\begin{Lemma}\label{lem:chordless}
Let \(N^{(1)}\) and \(N^{(2)}\) are two connecting path between
an exceptional pair of odd cycles \(C_1\) and \(C_2\).
If \(N^{(1)+}\) properly divides \(N^{(2)+}\),
then 
\(\theta_{1 2}N^{(2)+} - C_1^- C_2^- N^{(2)-}\)
is not crude.
\end{Lemma}
\begin{proof}
  Obviously, both \(\theta_{1 2}N^{(1)+} - C_1^- C_2^- N^{(1)-}\)
  and \(\theta_{1 2}N^{(2)+} - C_1^- C_2^- N^{(2)-}\) are in
  \(I_{\hat{G}}\).
  Since \(N^{(1)+}\) divides \(N^{(2)+}\),
  path \(N^{(2)}\) branches at some vertex \(u\) from path \(N^{(1)}\)
  but joins again at the vertex \(v\) just an edge apart from \(u\)
  along with \(N^{(1)}\).
  Moreover, the next edge is shared by both half paths,
  the number of edges on the subpath \(P\) of \(N^{(2)}\) from \(u\) to \(v\)
  is odd.
  Then, the edge \(e\) on \(N^{(1)}\) connecting \(u\) and \(v\)
  forms an even cycle with the subpath \(P\).
  The even cycle corresponds to an element in \(I_G\):
  \(P^+ e - P^-\), where \(P^{\pm}\) are restrictions of \(N^{(2)\pm}\)
  on \(P\).
  % There may be such branches and joins multiple times, but for simplicity,
  % we show the case only one such branch and join.
  By Lemma~\ref{lem:commondivisor}, the existence of \(e\) as a common divisor
  of \(C_1^- C_2^- N^{(1)-}\) and \(P^+ e\) is sufficient to conclude that
  \(\theta_{1 2}N^{(2)+} - C_1^- C_2^- N^{(2)-}\)
  is not crude.
\end{proof}

In Section~\ref{sec:second}, the lemma above will be generalized,
but we continue the proof of the theorem. 

\begin{Lemma}\label{lem:basics}
  The following elements in the ideal \(I_{\hat{G}}\) are crude.
  \begin{enumerate}
  \setlength{\itemsep}{1pt}
  \item\label{b1} \(\theta_{i j}^2 -\ C_i C_j\) for any \(\theta_{i j} \in \Theta\);
  \item\label{b2} \(\theta_{i j} N_{i j}^{(p)\pm} -\ C_i^{\mp} C_j^{\mp} N_{i j}^{(p)\mp}\) for any \(\theta_{i j} \in \Theta\) and \(N_{i j}^{(p)\pm}\) without \(N_{i j}^{(q)\pm}\) which properly divides \(N_{i j}^{(p)\pm}\);
  \item\label{b3} \(\theta_{i j} N_{j k}^{(p)\pm} C_k^{\pm} -\ \theta_{i k} N_{j k}^{(p)\mp} C_j^{\mp}\) for any \(\theta_{i j} \in \Theta\) and \(N_{i j}^{(p)\pm}\) without \(N_{i j}^{(q)\pm}\) which properly divides \(N_{i j}^{(p)\pm}\);
  \item\label{b4} \(\theta_{i j}\theta_{k l} - \tilde{\theta}_{i k}\tilde{\theta}_{j l}\) for any \(\theta_{i j}, \theta_{k l} \in \Theta\) with \(i,j,k,l\) are different each other; and
  \item\label{b5} \(\theta_{i j}\theta_{i k} - \tilde{\theta}_{j k}C_i\) and  \(\theta_{i j}\theta_{l j} - \tilde{\theta}_{i l}C_j\) for any \(\theta_{i j}, \theta_{i k}, \theta_{l j} \in \Theta\).
  \end{enumerate}
  Here, \(\tilde{\theta}_{i j}\) means either \(\theta_{i j}\) if \(C_i\) and \(C_j\) are an exceptional pair or \(C_i^{\pm}C_j^{\pm}\) otherwise.
\end{Lemma}
\begin{proof}
  (\ref{b1}) Since \(\theta_{i j}\) is an irreducible element,
  there is no monomial \(T\) in \(K[\monoid{F}]\) other than itself
  that \(\theta_{i j} \equiv T \pmod{I_{\hat{G}}}\).
  The proper divisor of \(\theta_{i j}^2\) is only \(\theta_{i j}\);
  thus, \(\theta_{i j}^2 - C_i C_j\) is crude.

  (\ref{b2}) Assume the contrary that 
  \(\theta_{i j} N_{i j}^{(p)+} - C_i^{-} C_j^{-} N_{i j}^{(p)-}\)
  is not crude.
  Then, there exists a proper divisor \(T \in K[\monoid{F}]\)
  of \(\theta_{i j} N_{i j}^{(p)+}\), which is congruent to some \(U\).
  As in the argument of (\ref{b1}), \(T\) cannot be \(\theta_{i j}\).
  Thus, there is a divisor \(D\) of \(T\) divides \(N_{i j}^{(p)+}\).
  The degree of \(D\) is in a range \(1\) to \(\deg(N_{i j}^{(p)+}) - 1\).
  Thus, the number of edges in \(N_{i j}^{(p)+}\) is more than one.
  Hence, there are edges \(e^{(p)}_{2k}\) in \(N_{i j}^{(p)+}\)
  and \(e^{(p)}_{2k+1}\) in \(N_{i j}^{(p)-}\), where
  \(\psi(e^{(p)}_{\kappa}) = v^{(p)}_{\kappa} v^{(p)}_{\kappa+1}\).
  Suppose \(e^{(p)}_{2k}\) does not divide \(D\).
  Then, \(v^{(p)}_{2k+1}\) does not divide \(\psi^*(T)\).
  As assumed, \(T \equiv U \pmod{I_{\hat{G}}}\), neither \(e^{(p)}_{2k}\) 
  nor \(e^{(p)}_{2k+1}\) divides \(U\), and \(v^{(p)}_2\) does not divide
  \(\psi^*(U)\).
  This argument continues until all edges in \(N_{i j}^{(p)}\)
  are excluded, or we find a short cut path directly connecting 
  \(v^{(p)}_{2k}\) to \(v^{(p)}_{2k+2l+1}\).
  The former contradicts with the existence of the divisor \(T\),
  and the latter contradicts with the assumption that there is no
  dividing path from Lemma~\ref{lem:chordless}.
  Therefore,
  \(\theta_{i j} N_{i j}^{(p)\pm} - C_i^{\mp} C_j^{\mp} N_{i j}^{(p)\mp}\)
  is crude.

  We omit the rest of the cases; 
  the proof of (\ref{b3}) is similar to that of (\ref{b2}),
  while the proofs of (\ref{b4}) and (\ref{b5}) are similar to (\ref{b1}).
\end{proof}

Before proving the next lemma, we should introduce some terminology.
A cycle \(C_i\) (and \(C_j\)) {\em semi-supports\/} \(\theta_{i j}\).
We define
a \(T\)-{\em induced subgraph\/} \(G'\) of \(G\) for \(T\), a monomial of
\(K[\monoid{F}]\) as a subgraph \(G'\) of \(G\) consisting of
every edge dividing \(T\) and
every edge of cycle \(C_i\) semi-supporting \(\theta_{i j}\) dividing \(T\).
Moreover, if a subgraph \(G''\) of \(G\) is
either a connected component with nonzero even semi-supporting cycles or
a pair of connected components both with odd semi-supporting cycles,
we call the subgraph \(G''\) an {\em even component};
it corresponds to a connected component of the hypergraph \(\hat{G}\).

\begin{Lemma}\label{lem:evencomponents}
  If \(T - U\) is a crude element in the ideal
  \(I_{\hat{G}}\), then \(TU\)-induced subgraph of \(G\) has at most
  two disjoint even components.
\end{Lemma}
\begin{proof}
  Without loss of generality, we can assume \(TU\) is not divisible by
  any product of all edges in an even closed walk.
  Let \(G'\) denote the \(TU\)-induced subgraph of \(G\).

  Assume that \(G'\) has three disjoint even components
  \(G_0\), \(G_1\) and \(G_2\).
  Then, by arranging \(\theta\) in \(T\) and \(U\)
  with (\ref{b4}) or (\ref{b5}) of Lemma~\ref{lem:basics},
  we have \(T'\equiv T\) and \(U'\equiv U \pmod{I_{\hat{G}}}\)
  satisfying the following condition:
  if an odd cycle in an even component \(G_i\) semi-supports a \(\theta\),
  the other cycle semi-supporting the same \(\theta\) is also in \(G_i\)
  for both \(T'\) and \(U'\).
  Then, each of \(T'\) and \(U'\) is decomposed into \(G_i\) parts
  \(T'_i\) and \(U'_i\) respectively for \(i=0,1,2\)
  and possibly a \(G'\setminus (\bigcup G_i)\) part.

  The decomposition implies that \(T' - U'\) is not primitive.
  Therefore, \(T - U\) is not crude:
  this is a contradiction.
\end{proof}

\begin{Lemma}\label{lem:nexus}
  If \(T - U\) is in the ideal \(I_{\hat{G}}\) 
  but not in \(I_G\) and
  an \(N_{i j}^{(p)\pm}\) divides \(T\),
  then \(T - U\) is not a crude element
  unless itself is one of (\ref{b2}) and (\ref{b3}) of Lemma~\ref{lem:basics}.
\end{Lemma}
\begin{proof}
  Assume that \(T - U\) is not one of (\ref{b2}) and (\ref{b3}) of
  Lemma~\ref{lem:basics}, and that \(N_{1 2}^{(1)+}\) divides \(T\)
  but \(N_{1 2}^{(1)+}\) does not divide \(U\).

  If both \(C_1\) and \(C_2\) semi-support \(\theta\)'s, then
  there exists \(X \neq 1\) such that
  \(T\equiv \theta_{1 2} N_{1 2}^{(1)+} X \pmod{ I_{\hat{G}}}\),
  by using (\ref{b4}) of Lemma~\ref{lem:basics} if necessary.
  Then, by (\ref{b2}) of Lemma~\ref{lem:basics}, we have
  \(T\equiv C_1^{-}C_2^{-} N_{1 2}^{(1)-} X \pmod{ I_{\hat{G}}}\).
  Since \(N_{1 2}^{(1)+}\) does not divide \(U\),
  \(N_{1 2}^{(1)-}\) must divide \(U\).
  Thus, Lemma~\ref{lem:commondivisor} implies that \(T - U\) is not crude.

  The case when both \(C_1\) and \(C_2\) do not semi-support \(\theta\)'s
  is just a reverse course of the above.

  We, then assume that \(C_1\) semi-supports a \(\theta\) but \(C_2\) does not.
  Then, there exists \(X \neq 1\) such that
  \(T \equiv \theta_{1 3} C_2^{+} N_{1 2}^{(1)+} X \pmod{I_{\hat{G}}}\).
  Then, by (\ref{b3}) of Lemma~\ref{lem:basics}, we have
  \(T\equiv C_1^{-}\theta_{2 3} N_{1 2}^{(1)-} X \pmod{I_{\hat{G}}}\).
  Since \(N_{1 2}^{(1)+}\) does not divide \(U\),
  \(N_{1 2}^{(1)-}\) must divide \(U\).
  Thus, Lemma~\ref{lem:commondivisor} implies that \(T - U\) is not crude.
\end{proof}

The above lemmata lead us to the main theorem.

\begin{Theorem}\label{thm:main}
  The following elements form a generating set of \(I_{\hat{G}}\):
  \begin{enumerate}
  \setlength{\itemsep}{1pt}
  \item\label{t1} a set of crude generators of \(I_G\);
  \item\label{t2} \(\theta_{i j}^2 - C_i C_j\) for any \(\theta_{i j} \in \Theta\);
  \item\label{t3} \(\theta_{i j} N_{i j}^{(p)\pm} - C_i^{\mp} C_j^{\mp} N_{i j}^{(p)\mp}\) for any \(\theta_{i j} \in \Theta\) and \(N_{i j}^{(p)\pm}\) without \(N_{i j}^{(q)\pm}\) which properly divides \(N_{i j}^{(p)\pm}\);
  \item\label{t4} \(\theta_{i j} N_{j k}^{(p)\pm} C_k^{\pm} - \theta_{i k} N_{j k}^{(p)\mp} C_j^{\mp}\) for any \(\theta_{i j} \in \Theta\) and \(N_{i j}^{(p)\pm}\) without \(N_{i j}^{(q)\pm}\) which properly divides \(N_{i j}^{(p)\pm}\);
  \item\label{t5} \(\theta_{i j}\theta_{k l} - \tilde{\theta}_{i k}\tilde{\theta}_{j l}\) for any \(\theta_{i j}, \theta_{k l} \in \Theta\) with \(i,j,k,l\) are different each other; and
  \item\label{t6} \(\theta_{i j}\theta_{i k} - \tilde{\theta}_{j k}C_i\) and \(\theta_{i j}\theta_{l j} - \tilde{\theta}_{i l}C_j\) for any \(\theta_{i j}, \theta_{i k}, \theta_{l j} \in \Theta\).
  \end{enumerate}
  Here, \(\tilde{\theta}_{i j}\) means either \(\theta_{i j}\) if \(C_i\) and \(C_j\) are an exceptional pair or \(C_i^{\pm}C_j^{\pm}\) otherwise.
\end{Theorem}
\begin{proof}
  Lemmata~\ref{lem:basics} through \ref{lem:nexus} determine
  the crude elements in the ideal \(I_{\hat{G}}\).
  That is, besides elements from \(I_G\) of \eqref{t1},
  if an \(N_{i j}^{(p)\pm}\) appears in a crude element,
  that element is of \eqref{t3} or \eqref{t4} by Lemma~\ref{lem:nexus};
  otherwise there are only pure odd cycles with at most four cycles
  by Lemma~\ref{lem:evencomponents}.
  Thus, the element is of \eqref{t2}, \eqref{t5} or \eqref{t6}.
  From Proposition~\ref{prop:spgenerate},
  these crude elements generate the ideal \(I_{\hat{G}}\).
  All monomials appearing in \eqref{t2} through \eqref{t6} are not divisible
  by any monomial appearing in \eqref{t1}.
\end{proof}

\subsection{Algorithm}
\label{sec:algorithm}

We summarize the argument above
as an algorithm to obtain the Ehrhart
polynomial from a given connected simple graph.

\begin{Algorithm}\label{algo}
Input is a connected simple graph \(G\), and
Output is the Ehrhart polynomial \(i_G(t)\).
\begin{enumerate}
  \setlength{\itemsep}{1pt}
\renewcommand{\labelenumi}{\arabic{enumi}.}
\item\label{step1} List all simple cycles in the given graph \(G\).
\item\label{step2} List all odd cycles among cycles of step~\ref{step1}.
\item\label{step3} List all paths connecting each pair of odd cycles.
\item\label{step4} Construct even closed walks of 
  Theorem~\ref{thm:evenclosedwalks} from the collected data.
\item\label{step5} Prepare variables for each exceptional pair, and
  construct all five types of generators of Lemma~\ref{lem:basics}.
\item\label{step6} Compute a \groebner basis of the ideal generated by
  the ring elements corresponding to the ideal elements of 
  step~\ref{step4} and~\ref{step5}.
\item\label{step7} From the initial terms of the \groebner basis of 
  step~\ref{step6}, obtain the Ehrhart series.
\item\label{step8} From the Ehrhart series of step~\ref{step7}, obtain and
  output the Ehrhart polynomial.
\end{enumerate}
\end{Algorithm}

It may be obvious to the readers that step~\ref{step5} is only required when
the given graph is non-edge-normal.
Be cautious that this algorithm is not efficient for general graphs,
which may have many odd cycles or many paths between them,
though we expect it is useful for many purposes.

\begin{Example}
\label{eg:twotriangles}

The example below illustrates how to use the algorithm.

%\subsubsection{Bow tie}
Let \(G\) be a graph %drawn below:
with
\(E = \Set{ e_0,\ldots,e_7 }\), \(V = \Set{ v_0,\ldots,v_6 }\), and
the correspondence as \(e_0 \mapsto v_0 v_1\),
\(e_1 \mapsto v_1 v_2\),
\(e_2 \mapsto v_2 v_0\),
\(e_3 \mapsto v_0 v_3\),
\(e_4 \mapsto v_3 v_4\),
\(e_5 \mapsto v_4 v_5\),
\(e_6 \mapsto v_5 v_6\), and 
\(e_7 \mapsto v_4 v_6\).
The graph is so-called ``bow-tie'', that
there are two triangles (cycles of length three) connected by a path
of length two; thus, the graph is non-edge-normal.
% Since it is not a bipartite graph,
% \(\dim\mathcal{P}_G = \numberof{V} - 1 = 5\).

There are four generators in the ideal \(I_{\hat{G}}\):
(G1) a type (\ref{t1}): \(e_0 e_2 e_4^2 e_6 - e_1 e_3^2 e_5 e_7\),
(G2) a type (\ref{t2}): \(\theta^2 - e_0 e_1 e_2 e_5 e_6 e_7\),
(G3) a type (\ref{t3}): \(\theta e_3 - e_0 e_2 e_4 e_6\), and
(G4) another type (\ref{t3}): \(\theta e_4 - e_1 e_3 e_5 e_7\).
It is easy to see that these immediately correspond to a \groebner
basis of the ideal \(I_{\hat{G}}\) for, say, lexicographic order
\(\theta > e_0 > \ldots > e_7\).
Then, the Ehrhart series is obtained through a multivariate generating function
as explained in Section~\ref{sec:moebius} just below.
Let \(\hat{H}_G\) denote the generating function obtained:

\begin{eqnarray*}
  &&
  \hat{H}_G(e_0,\ldots,e_7,\theta)\\
  &=&
  \frac{1
    -e_0 e_2 e_4^2 e_6
    -\theta^2
    -\theta e_3
    -\theta e_4
    +\theta e_0 e_2 e_4^2 e_6
    +\theta e_3 e_4
    +\theta^2 e_3
    +\theta^2 e_4
    -\theta^2 e_3 e_4
  }{(1-\theta)\prod_{i=0}^7 (1-e_i)}\\
  % &=&
  % \frac{(1 - \theta)
  % -e_0 e_2 e_4^2 e_6 (1-\theta)
  % +\theta(1-\theta)(1 - e_3)(1 - e_4)
  % }{(1-\theta)\prod_{i=0}^7 (1-e_i)}\\
  &=&
  \frac{1 - e_0 e_2 e_4^2 e_6 + \theta(1-e_3)(1-e_4)}{\prod_{i=0}^7 (1-e_i)}.
\end{eqnarray*}

Since the \(\psi^*\)-degree of each \(e_i\) is one and of \(\theta\) is three,
substituting \(t\) for each \(e_i\) and \(t^3\) for \(\theta\)
gives the Ehrhart series \(H_G\).

\[
  H_G(t) = \hat{H}_G(t,\ldots,t,t^3)
  =
  \frac{1-t^5 +t^3(1-t)^2}{(1-t)^8}\\
  =
  \frac{1 + t + t^2 + 2 t^3}{(1-t)^{7}}.
\]

(Notice the difference from the Hilbert series for the edge ring \(K[G]\):
\(\displaystyle  \frac{1 + t + t^2 + t^3 + t^4}{(1-t)^{7}}\).)

Then, the Ehrhart polynomial \(i_G(m)\) is

 \begin{eqnarray*}
   i_G(m)
   &=&
   \binom{m + 6}{6} + \binom{m + 5}{6} + \binom{m + 4}{6} + 2\binom{m + 3}{6}\\
   &=&\frac{1}{720}(m+3)(m+2)(m+1)\left(5m^3 + 21 m^2 + 94 m + 120\right)
 \end{eqnarray*}

\end{Example}

\section{Factoring Properties}
\label{sec:factoring}

\subsection{\moebius sum on lcm-lattice}
\label{sec:moebius}

There may be various methods to compute Ehrhart series on step~\ref{step7}
of Algorithm~\ref{algo},
but we use multivariate series as a convenient tool.
By Macaulay's theorem, the dimension of \(K[\hat{G}]_n\) can be 
computed by counting the monomials of degree \(n\) outside 
the initial ideal.
The main part of the computation is the \moebius sum
on lcm-lattice, which is a lattice on all \(\lcm\)s of monomials ordered by
divisibility~\cite{GPW1999}.
The lcm-lattice of our case is defined on initial monomials
\(\Set{f_i = \initialord{g_i} | g_i \in \Gamma}\) of a \groebner basis
\(\Gamma\) with respect to a term order \(<\);
the elements of the lattice are least common multiples of initial monomials
with \(1\) as the bottom element (the least common multiple of empty set).
Let \(L(X)\) denote the lcm-lattice on atoms \(X=\Set{\xi_1,\ldots,\xi_s}\).
Moreover, let \(M(L(X))\) denote the \moebius sum on \(L(X)\)
\[M(L(X)) = \sum_{x\in L(X)} \mu(x) x,\]
where \(\mu(x) = \mu(1, x)\)
is the \moebius function on \(L(X)\) of interval \([1, x]\) (see
\cite{Stanley1986}, for example).
It is used to obtain a multivariate generating function \(\hat{H}_G\):
\[\hat{H}_G(\tau_1,\ldots,\tau_s) = \frac{M(L(\initialord{\Gamma}))}
  {\prod_{i=1}^{s}(1-\tau_i)},\]
where \(\tau_i\) denote each elements of \(F\).
Finally, substituting
%\(t^{\deg{f_i}}\) to each \(f_i \in \initialord{\Gamma}\) and 
\(t^{\deg{\tau_i}}\) to each \(\tau_i \in F\)
gives the Ehrhart series \(H_G\).

In the following sections, the factoring properties of the Ehrhart series
are discussed based on the factorization of the \moebius sum.

\begin{Lemma}\label{lem:fct}
  Let \(X\) and \(Y\) be two finite sets such that any pair
  \(\xi \in X\) and \(\eta \in Y\) are coprime.
  Then, the \moebius sum \(M(L(X \cup Y))\) can be factored as
  \(M(L(X \cup Y)) = M(L(X)) M(L(Y))\).
\end{Lemma}
\begin{proof}
  We claim that the \moebius function on a lcm-lattice is multiplicative;
  that is, \(\mu(1)=1\) and \(\mu(x y)=\mu(x)\mu(y)\) 
  if \(x \in X\) and \(y \in Y\).
  If this claim is valid, the lemma follows:
  \begin{eqnarray*}
    M(L(X \cup Y)) &=& \sum_{x \in X \land y \in Y} \mu(x y) x y\\
    &=& \sum_{x \in X \land y \in Y} \mu(x)\mu(y)x y\\
    &=& \sum_{x \in X}\mu(x)x \sum_{y \in Y}\mu(y)y\\
    &=& M(L(X)) M(L(Y)).
  \end{eqnarray*}

  Thus, we prove the claim.
  First, by definition, \(\mu(1)=1\).
  Second, assume that for any \(x' y' < x y\) the claim is correct.
  Then,
  \begin{eqnarray*}
    \mu(x y) &=& -\sum_{x' y' < x y} \mu(x' y')\\
    &=& -\sum_{x' y' < x y} \mu(x')\mu(y')\\
    &=& \mu(x)\mu(y) - \left(\sum_{x' \le x}\mu(x')\right) \left(\sum_{y' \le y}\mu(y')\right)\\
    &=& \mu(x)\mu(y).
  \end{eqnarray*}
  Finally, by induction on the lattice order, the claim holds.
\end{proof}

\begin{Lemma}\label{lem:spec}
  Let \(X\) be a finite set of monomials.
  Then, let \(U(t)\) be the univariate polynomial transformed from
  the \moebius sum \(M(L(X))\) on \(X\):
  forgetting the derivation of coefficients and
  substituting \(t^{\deg{x}}\) to every monomial \(x \in L(X)\).
  Then, \(U(t)\) is divisible by \((1-t)\).
\end{Lemma}
\begin{proof}
  Since \(U(t) = \sum_{x \in L(X)} \mu(x) t^{\deg{x}}\),
  \[U(1) = \sum_{x \in L(X)}\mu(x)1^{\deg{x}} = \sum_{x \in L(X)}\mu(x) = 0.\]
\end{proof}

\subsection{First Factoring Property}
\label{sec:first}

The first factoring property of the Ehrhart series corresponds,
roughly, to biconnected decomposition of a graph.
The main discrepancy presents with odd cycles,
which are always the most complicated part of the
discussion of Ehrhart series of edge polytopes.
We avoid digging deeper into the complications, parenthesize the
hard part as a whole.
Let \(B_1, \ldots, B_r\) be the biconnected decomposition
of a graph \(G\).  If there are odd cycle subgraphs in \(G\),
let \(G_0\) be the minimum connected subgraph containing all
biconnected components with odd cycle subgraphs of \(G\).
By renumbering, if necessary, we have a decomposition of \(G\) as
\(G_0, B_1, \ldots, B_{r'}\).
We call such decomposition the biconnected decomposition of \(G\)
{\em with oddments}.

We apply Lemma~\ref{lem:fct} to obtain the first factoring property.

\begin{Theorem}\label{thm:1st}
  The Ehrhart series \(H_G\) of a graph \(G\) has a factorization
  \[H_G(t) = H_{G_0}(t) \prod_{i=1}^{r'} H_{B_i}(t),\]
  where \(G_0\), \(B_1\), \ldots, \(B_{r'}\) are the biconnected decomposition
  of \(G\) with oddments.
\end{Theorem}
\begin{proof}
  From Theorem~\ref{thm:main}, the only patterns that odd cycles affect
  the Ehrhart series are in the oddments subgraph \(G_0\).

  Let \(\initialord{\Gamma}\) be the initial monomials of \groebner basis
  \(\Gamma\) of the ideal \(I_{\hat{G}}\) with respect to
  a term order \(<\).
  The Ehrhart series is obtained through the multivariate generating function:
  \(\hat{H}_G(\tau_1,\ldots,\tau_s) = 
  \displaystyle\frac{M(L(\initialord{\Gamma}))}
  {\prod_{\tau_i \in F}(1-\tau_i)}\),
  as in Section~\ref{sec:moebius}.

  We know a generating set of \(I_{\hat{G}}\) from Theorem~\ref{thm:main},
  but do not know a \groebner basis, explicitly.
  If an initial monomial of the generating set is in a decomposed component,
  then it is coprime to those in other components, since the non-initial
  monomial also in the same decomposed component with the initial monomial.
  Then, the monomial remains coprime to those from other components
  after the Buchberger algorithm by Buchberger's criterion.
  Therefore, we have a \groebner basis whose initial monomials are classified
  into each decomposed component.

  By Lemma~\ref{lem:fct}, the numerator of \(\hat{H}_G\)
  is factored along with the biconnected decomposition with oddments.
  The denominator is also factored, because each edge is
  classified into a decomposed component.

  The Ehrhart series is obtained from \(\hat{H}_G\)
  by substituting %\(t^{\deg{f_i}}\) to each \(f_i \in \initialord{\Gamma}\) and
  \(t^{\deg{\tau_i}}\) to each \(\tau_i \in F\).
\end{proof}

\subsection{Second Factoring Property}
\label{sec:second}

The second factoring property focuses on an edge.
As we have seen in Lemma~\ref{lem:chordless}, 
a chordal path can be separated into shortcut
path and an even cycle, if parity permits.
We generalize the property not only on a path but also on an even cycle.

A separating pair of vertices of a graph \(G\) is a pair of vertices
\(v_1\), \(v_2\) of \(G\) that the number of connected components
of \(G-\{v_1,v_2\}\) is greater than that of \(G\).
Let \(e\) be an edge of \(G\), and
\(v_1\) and \(v_2\) be the end vertices of the edge \(e\).
Then, \(G-\tilde{e}\) denotes \(G-\{v_1, v_2\}\).
We call an edge with its end vertices a {\em separating face},
if the number of connected components
of \(G-\tilde{e}\) is greater than that of \(G\).

\begin{Lemma}\label{lem:ebi}
  Let \(G\) be a biconnected graph.
  If a separating face (\(e\) with \(u\) and \(v\)) decomposes \(G\) 
  into at least two components one of which is bipartite,
  then, there is a generating set of the ideal \(I_{\hat{G}}\)
  having no cycles stretching over the bipartite component and
  another component.
\end{Lemma}
\begin{proof}
  By assumption, we have two decomposed components \(G^{(1)}\) and \(G^{(2)}\),
  one of which, say \(G^{(2)}\), is a bipartite subgraph of \(G\).
  In subgraphs \(G^{(i)}\) for both \(i=1,2\), the vertices \(u\) and \(v\)
  are degree at least \(2\); one of the adjacent edges is \(e\).
  Let \(A_u\) and \(A_v\) denote ones of the other edges adjacent to \(u\)
  in \(G^{(1)}\) and to \(v\) respectively, and similarly \(B_u\) and \(B_v\)
  in \(G^{(2)}\).

  Consider a big even cycle in \(G\) passing \(A_u, B_u\) and \(B_v, A_v\).
  By assumption of bipartiteness of \(G^{(2)}\), 
  if we numbers \(A_u\) the first and \(B_u\) the second on the cycle, 
  then the numbering of \(B_v\) is even and that of \(A_v\) is odd.
  Hence we can name the other edges on the cycle:
  \[A_1=A_u, B_2=B_u, A_3, B_4 \ldots, B_{2k}=B_v, A_{2k+1}=A_v, B_{2k+2}, \ldots, A_{2m-1}, B_{2m}.\]
  Then,
  \[\prod_{i=1}^{m} A_{2i-1} - \prod_{i=1}^{m} B_{2i}\]
  is in \(I_{\hat{G}}\).
  We claim that this is redundant in a generating set of the ideal.
  If the claim is valid, since the choice of even cycle is arbitrary,
  there is no need to include cycles stretching over both
  \(G^{(1)}\) and \(G^{(2)}\) in the generating set of the ideal.

  Now we prove the claim.
  Let \(A^{(i)} = \prod_{A_j \in E(G^{(i)})} A_j\) and
  \(B^{(i)} = \prod_{B_j \in E(G^{(i)})} B_j\) for \(i=1,2\). Then
  \[\prod_{i=1}^{m} A_{2i-1} - \prod_{i=1}^{m} B_{2i} = A^{(1)} A^{(2)} - B^{(1)} B^{(2)}.\]
  There are cycles in \(G^{(1)}\) and \(G^{(2)}\), each corresponds to
  \(A^{(1)} - B^{(1)} e\) and \(A^{(2)} e - B^{(2)}\) respectively:
  in other words, each half the big cycle with \(e\).
  Then,
  \[A^{(1)} A^{(2)} - B^{(1)} B^{(2)} = A^{(2)} (A^{(1)} - B^{(1)} e) + B^{(1)} (A^{(2)} e - B^{(2)}).\]
  Thus the binomial is generated by the small cycles, 
  one of which is in \(G^{(1)}\) and another in \(G^{(2)}\).
\end{proof}

Notice that \(A^{(1)} A^{(2)} - B^{(1)} B^{(2)}\) may be primitive, 
but never be crude.

\begin{Theorem}\label{thm:2nd}
  Let \(G\) be a connected graph and \(G^{(1)}\) and \(G^{(2)}\) be its subgraphs.
  Assume 
  (1) each edge of \(G\) belongs either \(G^{(1)}\) or \(G^{(2)}\),
  except exactly one edge \(e\) which is shared by both;
  (2) \(G^{(2)}\) is a bipartite graph; and
  (3) \(e\) is a part of a cycle in \(G^{(2)}\).
  Then, the Ehrhart series \(H_G(t)\) can be factored as
  \[H_G(t) = H_{G^{(1)}}(t) \left(H_{G^{(2)}}(t) (1-t)\right).\]
\end{Theorem}
\begin{proof}
  By Theorem~\ref{thm:1st}, we can assume that \(G^{(2)}\) is a
  biconnected graph.
  If \(G^{(1)}\) also is a biconnected graph, by Lemma~\ref{lem:ebi}, 
  there is a generating set consisting of binomials from each subgraph.
  Moreover, even if \(G^{(1)}\) is not a biconnected graph,
  the same argument applies on any cycles stretching over both subgraphs.
  Hence, the remaining concerns are connecting paths passing through \(G^{(2)}\)
  between odd cycles, both of which are in \(G^{(1)}\).
  However, the condition of Theorem~\ref{thm:main}(\ref{t3}) based on
  Lemma~\ref{lem:chordless} have already excluded such paths.

  Because \(G^{(2)}\) is bipartite, 
  one can chose a term order that the shared edge \(e\) does not appear
  in the initial terms.
  Thus, the same argument with in the Theorem~\ref{thm:1st} applies,
  i.e., initial monomials from different components are coprime then
  the \moebius sum is factored along with the decomposition.
  
  Finally, since the shared edge \(e\) is counted in both 
  \(H_{G^{(1)}}(t)\) and \(H_{G^{(2)}}(t)\), we should cancel a \((1-t)\)
  from the denominator of \(H_{G^{(2)}}(t)\).
\end{proof}

Note that both factoring properties are also applicable to the 
Hilbert series of edge rings.

\section{Bipartite Polygon Trees}
\label{sec:ept}

\subsection{Explicit Series}
\label{sec:ser}

We apply the factoring properties to a few families of graphs
to obtain explicit Ehrhart series for them.

Recall a {\em polygon tree} is a connected simple graph defined recursively
as follows (see \cite{KoMa2004}, for example).
A polygon, or a cycle, is a polygon tree.
If \(G\) is a polygon tree, then picking an edge of it and
make a new cycle graph \(G'\) share the edge with \(G\),
then resulting graph is a polygon tree.
We call a polygon tree a bipartite polygon tree, if all involving cycles
are even cycles.

Before proving the result of polygon trees,
let us recall the basic examples of the Ehrhart series.

\begin{Fact}\label{lem:bicon}
  The followings are well-known 
  Ehrhart series of a few biconnected graphs.
  \begin{enumerate}
  \setlength{\itemsep}{1pt}
  \item\label{i:edge} \(\displaystyle\frac{1}{1-t}\) if \(G\) is an edge;
  \item\label{i:cycle} \(\displaystyle\frac{1+t+\cdots+t^{n - 1}}{(1-t)^{2 n - 1}}\) 
    if \(G\) is an even cycle with \(2 n\) edges;
  \item\label{i:oddcycle} \(\displaystyle\frac{1}{(1-t)^{2 n - 1}}\) 
    if \(G\) is an odd cycle with \(2 n - 1\) edges.
  \end{enumerate}
\end{Fact}

\begin{Proposition}\label{prop:evenduct}
  The Ehrhart series \(H_G(t)\) for a bipartite polygon tree graph \(G\)
  with \(e\) edges and 
  \(f_{2 n}\) cycles with \(2 n\) edges for \(n \ge 2\) is:
  \begin{equation*}
    \label{eq:edes}
    H_G(t)=\frac{\prod (1+t+\cdots+t^{n - 1})^{f_{2 n}}}{(1-t)^{e - f}}
    \tag{*}
  \end{equation*}
  with \(f=\sum f_{2 n}\).
\end{Proposition}
\begin{proof}
  We show the proposition by induction on the number of even cycles.
  If the number of cycles is one, 
  the graph is an even cycle with \(2n\) edges, then,
  from Fact~\ref{lem:bicon}(\ref{i:cycle}), 
  the Ehrhart polynomial is
  \(\displaystyle\frac{1+t+\cdots+t^{n - 1}}{(1-t)^{2 n - 1}}\).
  It coincides with \(e=2n\) and \(f=1\) case of \eqref{eq:edes}, as desired.
  Assume that \eqref{eq:edes} is valid for bipartite polygon trees with
  \(f - 1\) cycles.
  Then, a polygon tree \(G\) consisting of \(e\) edges and \(f\) even cycles
  is considered as an even cycle \(C'\) of \(2n\) edges
  and a polygon tree \(G'\) of \(e - (2n - 1)\) edges and \(f - 1\) even 
  cycles sharing an edge.
  Since the sharing edge is a separating face,
  by Theorem~\ref{thm:2nd}, the Ehrhart series can be factored as
  \[ H_G(t) = H_{G'}(t)\left(H_{C'}(t)(1-t)\right). \]
  With the induction hypothesis and Fact~\ref{lem:bicon}(\ref{i:cycle}),
  the degree of denominator in total is 
  \[\left(e - (2n - 1)\right) - (f - 1) + (2n - 1) - 1 = e - f,\]
  and the numerator is in the form of \eqref{eq:edes}.
\end{proof}

Since the graph \(G\) of Proposition~\ref{prop:evenduct} is bipartite,
the dimension \(D\) of the edge polytope is \(v - 2\),
as shown in~\cite{OH1998}.
The degree of denominator is equals to \(D+1\) for any polytopes,
thus it should be \(v - 1\) in the current case.
Since the polygon trees are planar, the Euler characteristic of the graph
gives the equation \(v - e + f = 1\), i.e., \(v - 1 = e - f\),
which is equal to the degree of our formula.

Note that the formula \eqref{eq:edes} is valid by Theorem~\ref{thm:1st}
for bipartite graphs whose biconnected components are all polygon trees, 
including bipartite cacti.
Moreover, the formula \eqref{eq:edes} is also valid if a single odd cycle
is in a polygon tree;
since we can start the induction from the odd cycle, whose Ehrhart series
is known as Fact~\ref{lem:bicon}(\ref{i:oddcycle}).
Note also that since the outerplanar graphs are subfamily of the polygon tree,
if it is bipartite or with a single odd cycle as above, the formula
applies to these cases as well.

\begin{Example}\label{eg:ladders}
  Ladders \(L_k\) are Cartesian products \(K_2 \times P_k\), where
  \(K_2\) is the complete graph of order two and
  \(P_k\) is the path graph of order \(k\).
  It is an even outerplanar graph and thus a bipartite polygon tree graph,
  consisting of \(k-1\) squares.
  Thus, the Ehrhart series \(H_{L_k}(t)\) can be deduced from 
  Proposition~\ref{prop:evenduct}:
  \[H_{L_k}(t) = \frac{(1+t)^{k-1}}{(1-t)^{2 k - 1}}.\]
\end{Example}
\begin{Example}
  We know the Ehrhart series of the bow-tie (Example~\ref{eg:twotriangles})
  and the ladders (Example~\ref{eg:ladders}).
  Then, for any combined graphs of bow-tie and ladders, sharing a vertex or
  an edge, we know their Ehrhart series.
  In case sharing a vertex, it is given by the first factoring property
  (Theorem~\ref{thm:1st}) as
  \[H_G(t) = \frac{(1 + t + t^2 + 2 t^3)(1+t)^{k-1}}{(1-t)^{2 k + 6}}.\]
  In case sharing an edge, it is given by the second factoring property
  (Theorem~\ref{thm:2nd}) as
  \[H_G(t) = \frac{(1 + t + t^2 + 2 t^3)(1+t)^{k-1}}{(1-t)^{2 k + 5}}.\]
\end{Example}

\subsection{Root Distribution}
\label{sec:roots}

The root distribution of the Ehrhart polynomials can be obtained from the 
Ehrhart series without explicit computation of the polynomials themselves
in some cases.
We use the results of Rodriguez-Villegas~\cite{Rod2002}.
For an integer \(a\), let \(S_a\) be a set of non-zero polynomials
\(p(x)\) such that
\[p(x) = v(x) \prod_{i=1}^{a}(x + i)\]
where all roots of \(v(x)\) lie on \(\realpart{x} = -(a+1)/2\).
Then the following lemma holds.
\begin{Lemma}[\cite{Rod2002}]\label{lem:roots}
Let \(\alpha\) be a root of unity
and \(f \in S_{a}\) for some \(a \in \mathbb{Z}\). Then
  \[f(x - 1) - \alpha f(x) \in S_{a-1}.\]
\end{Lemma}

\begin{Proposition}\label{prop:ptroots}
  The Ehrhart polynomial \(i_G(m)\) for a bipartite polygon tree \(G\)
  with \(e\) edges and \(f_{2 n}\) cycles with \(2 n\) edges for \(n\ge 2\) is
  in \(S_{e-1-\sum n f_{2n}}\).  In other words, the roots of \(i_G(m)\)
  are negative integers or on \(\realpart{x} = -(e-\sum n f_{2 n})/2\).
\end{Proposition}
\begin{proof}
  Let \(E^-\) denote the negative shift operator. 
  Then \(f(x - 1) - \alpha f(x)\) can be rewritten as \((E^- - \alpha)f(x)\).
  The Ehrhart polynomial \(i_G(m)\) is related to \(i_G^*(t)\), the numerator
  of the Ehrhart series, as
  \begin{eqnarray*}
    i_G(m) &=& i_G^*(E^-)\binom{m + D}{D}\\
    &=& c_h \prod_{j=1}^{h}(E^- - \alpha_j)\binom{m + D}{D},
  \end{eqnarray*}
  where \(h\) is the degree, 
  \(\alpha_j\) are the roots and
  \(c_h\) is the leading coefficient of \(i_G^*(t)\).
  From Proposition~\ref{prop:evenduct}, all roots of \(i_G^*(t)\)
  are roots of unity, and the leading coefficient is \(1\).
  Moreover, notice that \(\binom{m + D}{D}\) is in \(S_{D}\).
  When applying each factor \(E^- - \alpha_j\), we track the roots using
  Lemma~\ref{lem:roots}.
  Then, the intermediate polynomials are in \(S_{D-1}\), \(S_{D-2}\), and so on,
  and finally the Ehrhart polynomial is in \(S_{D-h}\).
  As noted after the proof of Proposition~\ref{prop:evenduct}, \(D=e-f-1\).
  On the other hand, the degree of \(i_G^*(t)\) is \(h=\sum (n-1) f_{2 n}\).
  Since \(f = \sum f_{2 n}\),
  \[D - h = e-f-1-\sum (n-1) f_{2 n} = e - 1 - \sum n f_{2 n}\]
  as required.
\end{proof}

Remark that the roots are on the strip of Conjecture~\ref{conj:dstrip},
and in fact on the left half-plane part of the region.

\nocite{Sturmfels1996}
\bibliographystyle{habbrv}
\bibliography{ehrhart,monoid,graph}

\end{document}